\documentclass[12pt,leqno,a4paper]{amsart}
\usepackage{amsmath,amsthm,amsfonts,amssymb,latexsym, mathabx}
\usepackage{enumerate, enumitem, verbatim}
\usepackage{color, xcolor,here}

\newcommand{\FF}{{\mathbb{F}}}

\newcommand{\bC}{{\mathbf{C}}}
\newcommand{\bH}{{\mathbf{H}}}

\newcommand{\bT}{{\mathbf{T}}}

\newcommand{\bG}{{\mathbf{G}}}

\newcommand\bg[1]{\mathbf{#1}}

\newcommand{\bZ}{{\mathbf{Z}}}

\newcommand{\ZZ}{{\mathbb{Z}}}

\newcommand{\SSS}{\mathsf{S}}

\newcommand{\CCC}{\mathsf{C}}

\newcommand{\Irr}{{\operatorname{Irr}}}

\newcommand{\out}{{\operatorname{Out}}}
\newcommand{\OO}{{\operatorname{O}}}

\newcommand{\SL}{{\operatorname{SL}}}
\newcommand{\Sp}{{\operatorname{Sp}}}
\newcommand{\SO}{{\operatorname{SO}}}

\newcommand{\Bl}{{\operatorname{Bl}}}
\newcommand{\Soc}{{\operatorname{Soc}}}

\newcommand{\PSL}{{\operatorname{PSL}}}
\newcommand{\PSp}{{\operatorname{PSp}}}
\newcommand{\POmega}{{\operatorname{P\Omega}}}
\newcommand{\Ker}{\operatorname{Ker}}

\newcommand{\Syl}{{\operatorname{Syl}}}

\newcommand{\bO}{{\mathbf O}}
\newcommand{\bF}{{\mathbf F}}

\newcommand{\Sym}{\mathfrak{S}}
\newcommand{\sym}{\mathfrak{S}}
\newcommand{\alt}{\mathfrak{A}}

\newcommand{\tw}[1]{{}^#1\!}

\def\oh#1#2{{\bf O}_{#1}(#2)}

\def\zent#1{{\bf Z}(#1)}
\def\irr#1{{\rm Irr}(#1)}
\def\norm#1#2{{\bf N}_{#1}(#2)}
\def\cent#1#2{{\bf C}_{#1}(#2)}
\def\ker#1{{\rm Ker}(#1)}

\def\sbs{\subseteq}

\theoremstyle{theorem}
\newtheorem{thm}{Theorem}[section]

\newtheorem{lem}[thm]{Lemma}
\newtheorem{cor}[thm]{Corollary}
\newtheorem{prop}[thm]{Proposition}

\newtheorem{cond}[thm]{Condition}
\newtheorem{rmk}[thm]{Remark}

\newtheorem{thml}{Theorem}

\newtheorem{condl}[thml]{Condition}

\theoremstyle{remark}
\newtheorem{rem}[thm]{Remark}
\newtheorem{exmp}[thm]{Example}

\newcommand\wt[1]{\widetilde{#1}}

\newcommand{\GL}{\operatorname{GL}}

\newcommand\type[1]{\operatorname{#1}}

\newcommand{\Aut}{\operatorname{Aut}}
\newcommand{\aut}{\operatorname{Aut}}

\usepackage{hyperref}

\begin{document}

\title{Brauer's Problem 21 for Principal Blocks}

\author{Alexander Moret{\'o}}
\address[A. Moret{\'o}]{Departament d'{\`A}lgebra - Universitat de Val{\`e}ncia, 46100 Burjassot,
Val\`encia, Spain}
\email{alexander.moreto@uv.es}

 \author{Noelia Rizo}
\address[N. Rizo]{Departament d'{\`A}lgebra - Universitat de Val{\`e}ncia, 46100 Burjassot,
Val\`encia, Spain}
\email{noelia.rizo@uv.es}

\author{A. A. Schaeffer Fry}
\address[A. A. Schaeffer Fry]{Dept. Mathematics - University of Denver, Denver, CO 80210, USA; and 
Dept. Mathematics and Statistics - MSU Denver, Denver, CO 80217, USA}
\email{mandi.schaefferfry@du.edu}
  
\thanks{The authors thank the  Isaac Newton Institute
for Mathematical Sciences (INI) in Cambridge and the organizers of the Summer 2022 INI program Groups, Representations, and
Applications: New Perspectives, supported by EPSRC grant EP/R014604/1, where part of this work was completed.  The second and third-named authors also thank the National Science Foundation Grant No. DMS-1928930, which supported them while they were in residence at the Mathematical Sciences Research Institute in Berkeley, California, during the Summer of 2023.
The first and second-named authors are supported by Ministerio de Ciencia e Innovaci\'on (Grants PID2019-103854GB-I00  and PID2022-137612NB-I00 funded by MCIN/AEI/10.13039/501100011033 and ``ERDF A way of making Europe" ). The first-named author also acknowledges support by Generalitat Valenciana CIAICO/2021/163. The second-named author is supported by a CDEIGENT grant CIDEIG/2022/29 funded by Generalitat Valenciana. The third-named author also gratefully acknowledges support from the National Science Foundation, Award No. DMS-2100912, and her former institution, Metropolitan State University of Denver, which holds the award and allows her to serve as PI. She also thanks the first and second authors and the CARGRUPS research team at U. Valencia for a productive stay in March 2023. The authors thank  A. Mar\'oti and G. Navarro for many useful conversations on Theorem C}



\begin{abstract}
Problem 21 of Brauer's list of problems from 1963 asks whether for any positive integer $k$  there are finitely many isomorphism classes of groups that occur as the defect group of a block with $k$ irreducible characters. 
We solve this problem for principal blocks. Another long-standing open problem (from 1982)  in this area asks whether the defect group of a block with $3$ irreducible characters is necessarily the cyclic group of order $3$.
In most cases we reduce this problem to a question on simple groups that is closely related to the recent solution of Brauer's height zero conjecture. 
 \end{abstract}

\maketitle


\section{Introduction}

In Problem 21 of his famous list of open problems in representation theory, R. Brauer asks
whether for any positive integer $k$  there are finitely many isomorphism classes of groups that occur as the defect group of a block with $k$ irreducible characters (\cite{bra}). 
This is equivalent to the question of whether the order of a defect group can be bounded from above in terms of the number of irreducible characters in the block.
This conjecture was proved for solvable groups by B. K\"ulshammer \cite{kul1} in 1989 and then for $p$-solvable groups  \cite{kul2} in 1990. On the other hand, using E. Zelmanov's solution of the restricted Burnside problem,  it was proved by K\"ulshammer and G. R. Robinson that the  Alperin--McKay conjecture implies Brauer's Problem 21 \cite{Kulshammer-Robinson}. 
Hence, L. Ruhstorfer's recent solution of Alperin--McKay for $p=2$ \cite{ruh} implies that Brauer's Problem 21 holds for this prime. In the main result of this paper, we prove Brauer's Problem 21 for principal blocks.

\begin{thml} \label{thm:BP21principal}
Brauer's Problem 21 has an affirmative answer  for principal blocks for every prime.
\end{thml}

Recall that Landau's theorem  asserts that the order of a finite group is bounded from above in terms of the number of conjugacy classes. As pointed out by  Brauer \cite{bra}, Landau's argument 
provides the bound $|G|\leq 2^{2^{k(G)}}$. Brauer's Problem 3 asks for substantially better bounds. 
This problem has also generated a large amount of research. L. Pyber \cite{pyb} found an asymptotically substantially better bound, although it is still not known whether there exists a bound of the form $|G|\leq c^{k(G)}$ for some constant $c$. We refer the reader to \cite{bmt} for the best known bound as of the writing of this article. Note that Brauer's Problem 21 asks for a blockwise version of Landau's theorem. As Brauer did with Landau's theorem, 
it also seems interesting to ask for asymptotically good bounds for the order of a defect group in terms of the number of characters in the block. Our proof of Theorem \ref{thm:BP21principal} provides an explicit bound  that surely will be far from best possible. For almost simple groups, we obtain a better bound in Theorem \ref{thm:BP21almostsimple}.

\medskip

Given a Brauer $p$-block $B$ of a finite group $G$ with defect group $D$, we will write $k(B)$ to denote the number of irreducible complex characters in $B$. 
R. Brauer himself proved that if $k(B)=1$ then $D$ is the trivial group (\cite[Theorem 3.18]{nbook}). More than 40 years later, J. Brandt proved that if $k(B)=2$ then $D$ is the cyclic group of order $2$. 
However, despite a large amount of work  in the area in recent years, the conjecture remains open when $k(B)\geq 3$. It has been speculated since Brandt's \cite{Br82} 1982 paper that if $k(B)=3$ then the defect group is cyclic of order $3$. It seems that it was known to K\"ulshammer that this follows from the Alperin-McKay conjecture since 1990  \cite{kul2}. A proof of this fact appeared in \cite{KNST14}, where K\"ulshammer, G. Navarro, B. Sambale and P. H. Tiep formally state Brandt's speculation as a conjecture. 

We present a condition on quasisimple groups that would imply the K\"ulshammer-Navarro-Sambale-Tiep conjecture (that is, that $k(B)=3$ implies that the defect group is of size $3$).


\medskip

\begin{condl}\label{quasisimples2}
Let $p$ be an odd prime and let $S$ be a non-abelian simple group of order divisible by $p$.   We say that Condition B holds for $(S,p)$ if the following holds: let $K$ be a quasisimple group of order divisible by $p$ with center $Z$, a cyclic $p'$-group, and $K/Z=S$. Let $B$ be a non-principal faithful $p$-block of $K$ with $|{\rm cd}(B)|>1$ and let $D$ be a defect group of $B$, not cyclic and elementary abelian. Then there are at least 4 irreducible characters in $B$ not ${\rm Aut}(K)$-conjugate. 
\end{condl}

\begin{thml}\label{thm:kb3reduction}
Let $p$ be a prime. If $p$ is odd, suppose that Condition B holds for $(S,p)$ for all non-abelian composition factors $S$ of $G$. Then the   
K\"ulshammer-Navarro-Sambale-Tiep conjecture holds for $G$. 
\end{thml}

We remark that this reduction and Condition \ref{quasisimples2} have played an influential role in the recent solution of Brauer's height zero conjecture \cite{MNST}. In fact, the fundamental Theorem B of \cite{MNST} is a slightly weaker version of Condition \ref{quasisimples2}:
 it shows there always exist $3$ irreducible characters in $B$ not ${\rm Aut}(K)$-conjugate. 
In fact, we will see in Remark \ref{rem:counterexample} that this is tight. Although Condition \ref{quasisimples2} seems to hold in many situations, we will see an example of a family of simple groups for which Condition \ref{quasisimples2} does not hold, for $p=5$.

\medskip

In Section \ref{sec:sectionsimplesBP21}, we prove Brauer Problem 21 for the principal blocks of almost simple groups, which is used in Section \ref{sec:BP21proof} to prove Theorem \ref{thm:BP21principal}. In Section \ref{sec:kb3reduction}, we prove Theorem \ref{thm:kb3reduction}. We conclude the paper by discussing Condition \ref{quasisimples2} in Section \ref{sec:orbits}.

\section{BP21 for almost simple groups}\label{sec:sectionsimplesBP21}

The following is the main result of this section. 

\begin{thm}
\label{thm:BP21almostsimple}

Let $p$ be a prime. Let $S\leq A\leq \aut(S)$, where $S$ is a finite nonabelian simple group, and $p\mid |S|$. 
Let $P\in \Syl_p(A)$ and let $k:=k(B_0(A))$ be the number of irreducible complex characters in the principal block of $A$. 
Then we have:
\begin{enumerate}[label=(\alph*)]
\item\label{thm:BP21simplemain} $|P|\leq k^{2(k^2+2k)}$.
\item \label{thm:BP21simple} $|P\cap S|\leq k^{2k^2}$.
\end{enumerate}
\end{thm}

\medskip

Note that in the context of Theorem \ref{thm:BP21almostsimple}, any character in $\irr{B_0(S)}$ lies below some character in $\irr{B_0(A)}$ by \cite[Theorem (9.4)]{nbook}, so that $k(B_0(A))\geq k_{\aut(S)}(B_0(S))$, where we write $k_{\aut(S)}(B_0(S))$ for the number of distinct $\aut(S)$-orbits intersecting $\irr{B_0(S)}$.   

\begin{rem}\label{rem:k7}
We further remark that, by the results of \cite{KS21, RSV21, HSV23}, we may assume for Theorem \ref{thm:BP21almostsimple} that $k(B_0(A))\geq 7$.
\end{rem}

The following is the main result of \cite{HSF21}, from which we obtain a bound for $p$ in terms of the number of irreducible characters in a given principal block.
\begin{lem}[Hung--Schaeffer Fry]\label{lem:HSF}
Let $p$ be a prime and let $G$ be a finite group with $p\mid |G|$. Let  $B_0$ denote the principal $p$-block of $G$. Then
\[k(B_0)^2\geq 4(p-1).\]
In particular, $p\leq \frac{1}{4}k(B_0)^2+1\leq \frac{1}{2}k(B_0)^2$, with the last inequality strict  for $k(B_0)>2$.
\end{lem}

Next, we consider the case of cyclic Sylow subgroups.

\begin{lem}\label{lem:BP21cyclic}
Let $p$ be a prime and let $G$ be a finite group with $p\mid |G|$. Assume that a Sylow $p$-subgroup $P\in\Syl_p(G)$ is cyclic, and let $B_0$ denote the principal $p$-block of $G$. Then
\[|P|<k(B_0)^2.\]
\end{lem}
\begin{proof}
In this case, by Dade's theory of blocks with cyclic defect group \cite[Theorem 5.1.2]{cravenbook}, we have $k(B_0)=e+\frac{|P|-1}{e}$, where $e=l(B_0)$ is the number of irreducible $p$-Brauer  characters in $B_0$. Since $1\leq l(B_0)<k(B_0)$ (see \cite[Theorem 15.29]{isbook2}), this yields
\[|P|=k(B_0)e-e^2+1\leq k(B_0)e<k(B_0)^2,\]
as claimed.
\end{proof}

\subsection{Notation and Additional Preliminaries}

Let $q$ be a power of a prime. By a group of Lie type, we will mean a finite group obtained as the group $\bG^F$ of fixed points of a connected reductive algebraic group $\bG$ over $\bar{\FF}_q$ under a Steinberg morphism $F\colon \bG\rightarrow\bG$ endowing $\bG$ with an $\FF_q$-structure. In our situation of finite simple groups, we will often take $\bG$ to further be simple and simply connected, so that $\bG^F$ is, with some exceptions dealt with separately, the full Schur covering group of a simple group $S=\bG^F/\zent{\bG^F}$.

Writing $G=\bG^F$, we let $G^\ast$ denote the group $(\bG^\ast)^F$, where the pair $(\bG^\ast, F)$ is dual to $(\bG, F)$, with respect to some maximally split torus $\bT$ of $\bG$. Given a semisimple  element  $s\in G^\ast$ (that is, an element of order relatively prime to $q$), we obtain a rational Lusztig series $\mathcal{E}(G,s)$ of irreducible characters of $G$ associated to the $G^\ast$-conjugacy class of $s$. When $s=1$, the set $\mathcal{E}(G,1)$ is comprised of the so-called unipotent characters. Each  series $\mathcal{E}(G,s)$ contains so-called semisimple characters, and if $\cent{\bG^\ast}{s}$ is connected, there is a unique semisimple character, which we will denote by $\chi_s$.  

The following lemma will help us obtain many semisimple characters in the principal block.
Here, we write $\zent{G}=\zent{
G}_p\times\zent{G}_{p'}$, where $\zent{G}_p\in\Syl_p(\zent{G})$. 

\begin{lem}\label{lem:ss}
Let $p$ be a prime and let $G:=\bg{G}^F$ be a group of Lie type defined over $\FF_q$ with $p\nmid q$ and such that $\zent{\bg{G}}$ is connected or such that $p$ is good for $\bg{G}$ and $\cent{\bg{G}^\ast}{s}$ is connected. Let $s\in G^\ast$ be a semisimple element with order a power of $p$. Then the corresponding semisimple character $\chi_s\in\Irr(G)$ lies in the principal $p$-block $B_0(G)$ of $G$ and is trivial on $\zent{G}_{p'}$.
\end{lem}
\begin{proof}
The first statement, also noted in \cite[Theorem 5.1]{HSF21}, is due to Hiss \cite[Corollary 3.4]{hiss90}, and the second follows from \cite[11.1(d)]{bon06}.
\end{proof}

Throughout, for $q$ an integer and $p$ a prime not dividing $q$, we let $d_p(q)$ denote the order of $q$ modulo $p$ if $p$ is odd, and $d_2(q)$ is the order of $q$ modulo $4$.

\medskip

For the remainder of Section \ref{sec:sectionsimplesBP21}, we will let $G=\bG^F$ for $\bG$ a simple, simply connected reductive group and $F\colon \bG\rightarrow\bG$ a Steinberg endomorphism such that $G/\zent{G}$ is a simple group of Lie type. Further, we will address the case that $S$ is a simple group with an exceptional Schur multiplier (see \cite[Table 6.1.3]{GLS} for the list of such $S$), sporadic, or alternating separately in the proof of Theorem \ref{thm:BP21simple} below, and hence until then, we assume further that $\zent{G}$ is a nonexceptional Schur multiplier for the simple group of Lie type $S:=G/\zent{G}$. Let $\wt{S}$ denote the group of inner-diagonal automorphisms of $S$. 
 
\subsection{Exceptional Groups}

We first consider the exceptional groups, by which we mean the groups $S=\type{G}_2(q)$, $\tw{2}\type{B}_2(q^2)$, $\tw{2}\type{G}_2(q^2)$, $\type{F}_4(q)$, $\tw{2}\type{F}_4(q^2)$, $\tw{3}\type{D}_4(q)$, $\type{E}_6(q)$, $\tw{2}\type{E}_6(q)$, $\type{E}_7(q)$, and $\type{E}_8(q)$, when $p$ is a prime not dividing $q$.  

Let $P\in\Syl_p(S)$. Then  either $P$ may be identified with a Sylow $p$-subgroup of $G$ or $(p,\bG)\in\{(3,\type{E}_6), (2, \type{E}_7)\}$ and $|P|=|\hat{P}|/p$ with $\hat{P}$ a Sylow $p$-subgroup of $G$.

If $G$ is not of Suzuki or Ree type (i.e. $G$ is not one of $\tw{2}\type{B}_2(q^2)$, $\tw{2}\type{G}_2(q^2)$, or $\tw{2}\type{F}_4(q^2)$), let $e:=d_p(q)$ and let $\Phi_e:=\Phi_e(q)$ denote the $e$th cyclotomic polynomial in $q$. If $G$ is a Suzuki or Ree group, instead let $\Phi_e:=\Phi^{(p)}$ as in \cite[Section 8]{malle07}. In either case, let $p^b$ be the the highest power of $p$ dividing $\Phi_e$ and let $m_e$ denote the largest positive integer such that ${\Phi_e}^{m_e}$ divides the  order polynomial of $(\bG, F)$. 

From \cite[Theorem 4.10.2]{GLS}, we see that $\hat P$ contains a normal abelian subgroup $P_T\lhd \hat P$ such that $\hat P/P_T$ is isomorphic to a subgroup of the Weyl group $W=\norm{\bG}{\bT}/\bT$. We also have $\hat P=P_T$ if and only if $P_T$ is abelian (see \cite[Proposition 2.2]{malle14}).  Similarly, a Sylow $p$-subgroup of the dual group $G^\ast$ contains a group isomorphic to $P_T$.

\begin{prop}\label{prop:BP21except2}
Let $p=2$ and let $S$ be an exceptional group of Lie type as above with $2\nmid q$. Let $P\in\Syl_2(S)$ and write $k_0:=k_{\aut(S)}(B_0(S))$. Then $|P|\leq 2^{14+8k_0}$. (In particular, Theorem \ref{thm:BP21almostsimple}\ref{thm:BP21simple} holds in this case.)
 
\end{prop}
\begin{proof}
First consider the case $S=\tw{2}\type{G}_2(q^2)$ with $q^2=3^{2n+1}>3$. Then we have $|P|=8$, so the statement is clear. Hence we assume that $S$ is not of Suzuki or Ree type.
Let $H=G^\ast=(\bG^\ast)^F$ and notice that $S=[H, H]$ and $\zent{\bG^\ast}$ is connected. Then notice that the semisimple characters $\chi_s\in\Irr(H)$ of $H$ for $s\in H^\ast=G$ of $2$-power order lie in $B_0(H)$ by Lemma \ref{lem:ss}. Let $2^{b+1}\mid\mid(q^2-1)$ and note that $G$ contains an element of order $2^{b}$. For $1\leq i\leq b$, let $s_i\in G$ be of order $2^{i}$, so that the semisimple characters $\chi_{s_i}$ of $H$ for $1\leq i\leq b$ lie in $B_0(H)$. Further, since the $|s_i|$ are distinct, these lie in distinct $\aut(S)$-conjugacy classes, using \cite[Corollary 2.5]{NTT08}. Then choosing an irreducible constituent $\chi_{s_i}'$ on $S$ for each $i$, we obtain $b$ characters in $B_0(S)$ in distinct $\aut(S)$-classes. Considering in addition the trivial character, we obtain $k_0\geq b+1$.

On the other hand, letting $r$ be the rank of $\bG$, we have $r\leq 8$ and $|P_T|\leq (2^{b+1})^r\leq (2^{b+1})^8$ by the description of $P_T$ in \cite[Theorem 4.10.2]{GLS}. Further, $|\hat P/P_T|\leq |W|_2\leq 2^{14}$. Hence $|P|\leq |\hat P|\leq 2^{14}\cdot 2^{8k_0}$, as stated.  Recalling that $k\geq 7$ (see Remark \ref{rem:k7}), in the situation of Theorem \ref{thm:BP21almostsimple}\ref{thm:BP21simple} we have $|P|\leq 7^5 \cdot 7^{3k}\leq k^{5+3k}$.
\end{proof}

Now, when $p$ is odd, a similar argument can be used. However, we aim for a better bound. In this case, \cite[Theorem 4.10.2]{GLS} further tells us that $P_T$ has a complement $P_W$ in $\hat P$ and  we have
 $P_T\cong C_{p^b}^{m_e}$ unless $(p, G)= (3,\tw{3}\type{D}_4(q))$, in which case $P_T\cong C_{3^a}\times C_{3^{a+1}}$. In the following, let $W(\type{E}_8)$ denote the Weyl group $W$ obtained in the case that $\bG=\type{E}_8$.
 
\begin{prop}\label{prop:BP21except}
 
Let $S$ be an exceptional group of Lie type as above, and let $P\in \Syl_p(S)$ with $p$ an odd prime  not dividing $q$. Let $k_0:=k_{\aut(S)}(B_0(S))$. Then if $P$ is cyclic, we have $|P|< p^{k_0}$. Otherwise, we have 
\[|P|\leq  C_{ex}\cdot k_0^2\]
 for some constant $C_{ex}\leq 36|W(\type{E}_8)|^2$. 
 In particular, when $S\leq A\leq \Aut(S)$ with 
  $k(B_0(A))\geq 5$, this yields $|P|\leq k(B_0(A))^{k(B_0(A))^2}$ in either case.
\end{prop}
It should be noted, however, that in the last statement, $P\in\Syl_p(S)$, rather than $\Syl_p(A)$.
\begin{proof}
Keep the notation above. Suppose first that a Sylow $p$-subgroup of $G$ is abelian. If $P$ is cyclic, then $P=\hat{P}=P_T=C_{p^b}$. Here we may argue similarly to Proposition \ref{prop:BP21except2} to obtain $b<k_0$, and hence $|P|<p^{k_0}$.  Lemma \ref{lem:BP21cyclic} further yields the last statement in this case. Hence, we may assume that $P$ is not cyclic, so that $m_e\geq 2$. By the discussion preceeding \cite[Theorem 5.4]{HSF21}, we have 
\begin{equation}\label{eq:HSFbnd}
k_0\geq \frac{p^{b m_e}}{gd p^b|W_e|},
\end{equation} where $g$ is the size of the subgroup of $\out(S)$ of graph automorphisms, $d:=[\wt{S}:S]$ is the size of the group of diagonal automorphisms, and $|W_e|$ is the so-called relative Weyl group for a Sylow $\Phi_e$-torus of $G$. Since $d\leq 3$, $g\leq 2$, and $|W_e|$ is bounded by the size of the largest Weyl group for the types under consideration, $|W(\type{E}_8)|$,  we have 
\[k_0\geq \frac{p^{bm_e}}{6|W(E_8)| p^b}.\]
  
Notice that $p^{b(m_e-1)}\geq p^{bm_e/2}$ for $m_e\geq 2$. Then we have $\sqrt{|P|}\leq \sqrt{|\hat P|}\leq 6|W(E_8)|k
_0$, and hence the statement holds.

We now assume that $\hat P$ is nonabelian. By considering only the semisimple characters of $G$ corresponding to elements of $G^\ast$ found in a copy of $P_T$, the exact same arguments as in \cite[Section 5]{HSF21} yield that 
the bound \eqref{eq:HSFbnd} still holds in this case.  
 
By considering the degree polynomials, we see that in each case, we have $\sqrt{|P|}\leq p^{b(m_e-1)}$ again, except possibly if $G=\type{G}_2(q)$, $p=3$, $m_e=2$, and $|P|=p^{2b+1}$. Then $\sqrt{|P|}=\sqrt{3}\cdot 3^b= \sqrt{3}\cdot p^{b(m_e-1)}\leq \sqrt{3}|W_e|k_0$, where the last inequality is because $d=1=g$ in this case.
In all cases, then, we see that the statement holds.
\end{proof}

\subsection{Classical Groups}\label{sec:classical}

We now turn to the case of classical groups. In this section, let $G=\bG^F$ be a group of Lie type defined over $\FF_q$, where $q$ is a power of a prime $q_0$ and $\bG$ is a simple, simply connected reductive group of type $\type{A}_{n-1}$ with $n\geq 2$, $\type{C}_n$ with $n\geq 2$, $\type{B}_n$ with $n\geq 3$, or type $\type{D}_n$ with $n\geq 4$ but $G\neq \tw{3}\type{D}_4(q)$, and such that $G$ is a nonexceptional Schur covering group for the simple group $S:=G/\zent{G}$.   That is, $G=\SL_n^\epsilon(q)$, $\Sp_{2n}(q)$,  $\operatorname{Spin}_{2n+1}(q)$, or $\operatorname{Spin}_{2n}^\pm(q)$, and $S=\PSL_n^\epsilon(q)$, $\PSp_{2n}(q)$, $\POmega_{2n+1}(q)$, or $\POmega_{2n}^\pm(q)$, respectively, for the corresponding values of $n$. Let $H$ be the related groups $H:=\GL_n^\epsilon(q), \Sp_{2n}(q), \SO_{2n+1}(q)$, respectively $\SO_{2n}^\pm(q)$.  We remark that, taking $\Omega:=\OO^{q_0'}(H)$, we have $\Omega$ is perfect and $S=\Omega/\zent{\Omega}=G/\zent{G}$. We also have $\zent{\Omega}\leq \zent{H}$ and further $H/\Omega$ and $\zent{H}$ are both $2$-groups if $H\neq \GL_n^\epsilon(q)$.
 Note that the dual group of $H$ is $H^\ast = \GL_n^\epsilon(q)$, $\SO_{2n+1}(q)$, $\Sp_{2n}(q)$, and $\SO_{2n}^\pm(q)$, respectively.   
 
 Let $p\neq q_0$ be a prime and write $\wt{P}$ for a Sylow $p$-subgroup of $H^\ast$. We remark that if  $X\in\{G, S\}$, then $|P|\leq |\wt{P}|$ for $P\in\Syl_p(X)$.
 
\subsubsection{Sylow $p$-Subgroups of Symmetric Groups}

Since the Sylow $p$-subgroups of classical groups are closely related to those of symmetric groups, we begin with a discussion of the latter.

Let $w$ be a positive integer with $p$-adic expansion \begin{equation}\label{eq:wpadic}
w=a_0+a_1p+a_2p^2+\cdots+a_tp^t,
\end{equation}  where $0\leq a_i<p$ for $0\leq i\leq t-1$ and $0<a_t<p$.  Let $Q\in\Syl_p(\Sym_w)$. We have $Q=\prod_{i=0}^{t}Q_i^{a_i}$, , where $Q_i$ is a Sylow $p$-subgroup of the symmetric group $\Sym_{p^i}$. Moreover $|Q_i|=p^{p^{i-1}+p^{i-2}+\cdots+p+1}\leq p^{p^i}$ for each $1\leq i\leq t$. Then with this, we see 
\begin{equation}\label{eq:sizeSylowSw}
 |Q|=(w!)_p\leq p^w.
\end{equation}
 \subsubsection{Unipotent Characters}
 
 Recall that $\mathcal{E}(G,1)$ is the set of unipotent characters of $G$. Since unipotent characters are trivial on $\zent{G}$, we may say that $\chi\in\Irr(S)$ is a unipotent character of $S$ if it is the deflation of some unipotent character of $G$.  
 The following observation will be useful in the cases of defining characteristic and when $p=2$.
\begin{lem}\label{lem:classicalunips}
Let $S$ be one of the groups $S=\PSL_n^\epsilon(q)$ with $n\geq 2$, $\PSp_{2n}(q)$ with $n\geq 2$, $\POmega_{2n+1}(q)$ with $n\geq 3$, $\POmega_{2n}^+(q)$ with $n\geq 4$, or $\POmega_{2n}^-(q)$ with $n\geq 4$. Then there are at least $n$ non-$\aut(S)$-conjugate unipotent characters of $S$.
\end{lem}
\begin{proof}
The unipotent characters of $G$ are described in \cite[Section 13.8]{carter}. From this, we have the number of unipotent characters in the case $\PSL_n^\epsilon(q)$ is the number of partitions $\pi(n)$ of $n$. In the remaining cases, the unipotent characters of  $G$ lying in the principal series are in bijection with the characters of the Weyl group  $W(\type{C}_n)$, $W(\type{B}_n)$, $W(\type{D}_n)$, or $W(\type{B}_{n-1})$, respectively, each of which contains a quotient group isomorphic to a symmetric group $\sym_n$ (resp. $\sym_{n-1}$ in the case of $W(B_{n-1})$). In each of these cases, there are also non-principal series unipotent characters. Then the number of unipotent characters is more than $\pi(n)$ (resp. $\pi(n-1)$) in these cases.
 Note that $\pi(n)\geq n$, with strict inequality for $n\geq 4$, and that further $\pi(n)\geq 2n$ for $n\geq 7$.

 With the exception of $\PSp_{4}(q)$ with $q$ even and $\POmega_{2n}^+(q)$, all unipotent characters of the groups under consideration are $\aut(S)$-invariant (see \cite[Theorem 2.5]{Malle08}), and we see that there are at least $n$ such characters in each case. For $\PSp_{4}(q)$ with $q$ even, there are six unipotent characters, with two of them interchanged by the exceptional graph automorphism. For $\POmega_{2n}^+(q)$ with $n\geq 5$, we see there are at least $2n$ unipotent characters (explicitly for $n=5,6$, and since $\pi(n)\geq 2n$ for $n\geq 7$), and the $\aut(S)$-orbits have size at most 2. The group $\POmega_8^+(q)$ has 14 unipotent characters and the $\aut(S)$-orbits have size at most 3. In all cases, then, we see that there are at least $n$ $\aut(S)$-orbits of unipotent characters.
\end{proof}

\subsubsection{Bounds in the Case of Classical Groups for Defining Characteristic or $p=2$}
\begin{cor}\label{cor:nunips}
Let $S$ be one of the groups as in Lemma \ref{lem:classicalunips}. Assume that $p\mid q$ or that  $p=2$ and $q$ is odd.  Then $k_{\aut(S)}(B_0(S))\geq n$.
\end{cor}
\begin{proof}
In defining characteristic, we have $\irr{B_0(S)}=\irr{S}\setminus\{\mathrm{St}_S\}$, where $\mathrm{St}_S$ is the Steinberg character (see \cite[Theorem 6.18]{CE04}). In the case $p=2$ and $q$ is odd, we have $B_0(S)$ is the unique block containing unipotent characters, by \cite[Theorem 21.14]{CE04}. Then the statement follows from Lemma \ref{lem:classicalunips} and the fact that $\irr{S}$  contains   non-unipotent characters.
\end{proof}

\begin{lem}\label{lem:boundclassicalb2}
Let $S$ be as in Lemma \ref{lem:classicalunips} with $q$ odd, and let $2^{b+1}$ be the largest power of $2$ dividing $q^2-1$. Let $B_0(S)$ be the principal $2$-block of $S$. Then  $b+1\leq k_{\Aut(S)}(B_0(S))$.
\end{lem}
\begin{proof}
As before, let $S=G/\zent{G}$ with $G=\bG^F$ of simply connected type. Recall that $B_0(G)$ is the unique unipotent block of $G$ by \cite[Theorem 21.14]{CE04}, and hence $B_0(G)$ is exactly the union of  rational Lusztig series $\mathcal{E}(G,s)$ with $s\in G^\ast$ having order a power of $2$. From the structure of the Sylow $2$-subgroup of $G^\ast$ described in \cite[Theorem 4.10.2]{GLS} (see also \cite{carterfong}), we can see that the group $\bO^{q_0'}(G^\ast)$ contains an element of order $2^{b-1}$. For $1\leq i\leq b-1$, let $s_i\in \bO^{q_0'}(G^\ast)$ be of order $2^{i}$.  Then the semisimple characters $\chi_{s_i}$ for $1\leq i\leq b-1$ lie in $B_0(G)$ and are trivial on $\zent{G}$ by  the dual version of \cite[Proposition  11.4.12 and Remark 11.4.14]{DM20}. Further, these lie in distinct $\aut(S)$-conjugacy classes, using \cite[Corollary 2.5]{NTT08}. Combining with Lemma  \ref{lem:classicalunips}, we see $k_{\Aut(S)}(B_0(S))\geq b-1+n\geq b+1$.
\end{proof}

 \subsubsection{Bounds in the Case of Classical Groups for Nondefining Characteristic with $p$ Odd}\label{sec:classicalBP21}

Now let $p$ be an odd prime not dividing $q$ and let $d:=d_p(q)$. If $G=\SL_n^\epsilon(q)$, let $e:=d_p(\epsilon q)$, and otherwise let $e:=d_p(q^2)$. Further, let $b\geq 1$ be the largest integer such that $p^b$ divides $(\epsilon q)^e-1$, respectively $q^{2e}-1$.

We begin by discussing the Sylow $p$-subgroups $\wt{P}$ of $H^\ast$, which have been described by Weir \cite{weir}.

First, consider the case $H=\GL_n^\epsilon(q)$.  Let $n=ew+r$, where $r,w$ are positive integers with  $0\leq r<e$ and $w$ is written with $p$-adic expansion as in \eqref{eq:wpadic}.  
A Sylow $p$-subgroup of $H=\GL_n^\epsilon(q)$ is then of the form $\wt{P}=\prod_{i=0}^t P_i^{a_i}$, where $P_i\in \Syl_p(\GL_{ep^i}^\epsilon(q))$  is of the form $C_{p^b}\wr Q_i$. Hence $\wt{P}$ contains a subgroup of the form $\bar P=C_{p^b}^w$. Further, $\wt{P}\cap \SL_n^\epsilon(q)$ is a Sylow $p$-subgroup of $G=\SL_n^\epsilon(q)$.

Now consider $H=\Sp_{2n}(q)$, $\SO_{2n+1}(q)$, and $\SO_{2n}^\epsilon(q)$. The structure of $\wt{P}$ in these case builds off of the case of 
 linear groups above.  If $H^\ast\in\{\SO_{2n+1}(q), \Sp_{2n}(q)\}$, we have $\wt P$ is already a Sylow $p$-subgroup of $\GL_{2n+1}(q)$ (and hence of $\GL_{2n}(q)$) when $d$ is even, and are Sylow subgroups of the naturally-embedded $\GL_n(q)$ if $d$ is odd.  
 In particular, writing $n=ew+r$ with $r, w$ as before, we have $\wt{P}$ is again  
 of the form $\wt{P}\cong P_0^{a_0}\times\cdots\times P_t^{a_t}$, where each $P_i$ is a Sylow $p$-subgroup of $\GL_{dp^{i}}(q)$ (and hence again $P_i\cong C_{p^b}\wr Q_i$).  
 
 If $H^\ast=\SO_{2n}^\pm(q)$, then we have embeddings $\SO_{2n-1}(q)\leq H^\ast\leq \SO_{2n+1}(q)$, and $\wt{P}$ is a Sylow subgroup of either $\SO_{2n-1}(q)$ or $\SO_{2n+1}(q)$.  In this case, letting $m\in\{n, n-1\}$ so that $\wt{P}$ is a Sylow subgroup of $\SO_{2m+1}(q)$ and now writing $m=ew+r$ with $w$ again written as in \eqref{eq:wpadic}, $\wt{P}$ can again be written $\wt{P}\cong P_0^{a_0}\times\cdots\times P_t^{a_t}$ with each $P_i$ a Sylow subgroup of $\GL_{dp^{i}}(q)$.

In all cases, we remark that $p^t\leq w\leq p^{t+1}$ and that $t=0$ corresponds to the case that a Sylow $p$-subgroup of $H$ is abelian. Further, $\wt{P}$ contains a subgroup of the form $\bar P\cong C_{p^b}^w$.

\begin{lem}\label{lem:boundclassicalb}
With the notation above, we have $b+1\leq k_{\Aut(S)}(B_0(S))$.
\end{lem}
 \begin{proof}
 
We will show that there are at least $b$ characters in $\irr{B_0(S)}\setminus\{1_S\}$ lying in distinct $\aut(S)$-orbits. 

First, let  $G=\SL_n^\epsilon(q)$ and let $\wt{G}:=H=\GL_n^\epsilon(q)$, and note $\wt{G}^\ast\cong \wt{G}$. We have $\aut(S)=\wt{S}\rtimes \mathcal{D}$, where $\mathcal{D}$ is an appropriate group of graph and field automorphisms and $\wt{S}:=\wt{G}/\zent{\wt{G}}$. Recall that a Sylow $p$-subgroup $\wt{P}$ of $\wt{G}$ contains a subgroup of the form $C_{p^b}^w$.

Assume for the moment that $e>1$, so that $p\nmid|\zent{\wt{G}}|$ and $p\nmid [\wt{G}:G]$.  Hence, for $1\leq j\leq b$, we may let $s_j\in \wt{G}^\ast\cong \wt{G}$ be an element of order $p^j$. The corresponding semisimple character $\chi_{s_j}$ of $\wt{G}$ is trivial on $\zent{\wt{G}}$ and lies in $B_0(\wt{G})$, using Lemma \ref{lem:ss}. 
Hence, each $\chi_{s_j}$ can be viewed as a character of $B_0(\wt{S})$. Further, note that since $p\nmid |\zent{\wt{G}}|$,  
$s_i$ and $s_j^\alpha z$ cannot be $\wt{G}$-conjugate for any $i\neq j$ and any $\alpha\in \mathcal{D}$ and $z\in\zent{\wt{G}}$.  

If instead $e=1$, we have $n=w\geq 2$. For $1\leq j\leq b$, let $\lambda_j\in C_{p^b}\leq \FF_{q^2}^\times$ with $|\lambda_j|=p^j$, and let $s_j$ be an element of $C_{p^b}^n\leq \wt{P}$ of the form $\mathrm{diag}(\lambda_j, \lambda_j^{-1}, 1, \ldots 1)$, where $1$ appears as an eigenvalue with multiplicity $n-2$.  Then again $\chi_{s_j}\in\irr{B_0(\wt{G})}$ by Lemma \ref{lem:ss} and is trivial on $\zent{\wt{G}}$ by the dual version of \cite[Proposition  11.4.12, and Remark 11.4.14]{DM20}, since $s_j\in [\wt{G},\wt{G}]=G$.  Further, we again see that $s_j$ is not $\wt{G}$-conjugate to $s_i^\alpha z$ for any $i\neq j$,  $\alpha\in \mathcal{D}$, and $z\in \zent{\wt{G}}$, by considering the eigenvalues. 

In either case, we let $\chi_i$ for $1\leq i\leq b$ be a constituent of $\chi_{s_i}$ restricted to $S$. Then  $\chi_i$ cannot be $\aut(S)$-conjugate to $\chi_j$ for $i\neq j$, using \cite[Corollary 2.5]{NTT08} along with \cite[ Proposition  11.4.12 and Remark 11.4.14]{DM20}. Hence, we see at least $b$ distinct $\aut(S)$-orbits represented in $\irr{B_0(S)}\setminus\{1_S\}$.

Now let $G$ be one of the remaining groups as in the beginning of the section. Then $|\zent{G}|$ is a power of $2$. In each case, a Sylow $p$-subgroup of $G$ (or, equivalently, of $S$) and of $G^\ast$ contains a subgroup of the form $C_{p^b}$. Here, we may again, for each $1\leq j\leq b$, let $s_j\in G^\ast$ be a semisimple element of order $p^j$. 
Then since each $(|s_j|, |\zent{G}|)=1$, we have by \cite[Exercise 20.16]{MT11} that $\cent{\bG^\ast}{s_j}$ is connected since $p\geq 3$ is good for $\bG$, and hence the corresponding semisimple character $\chi_{s_j}$ of $G$ lies in $B_0(G)$ and is trivial on $\zent{G}$ by Lemma \ref{lem:ss}. That is, we may again view $\chi_{s_j}$ as a character in $\irr{B_0(S)}\setminus\{1_G\}$. Since $s_i$  cannot be $\aut(G^\ast)$-conjugate to $s_j$ for any $i\neq j$, we see $\chi_{s_i}$ and $\chi_{s_j}$ cannot be $\aut(S)$-conjugate as before and we again have $k_{\aut(S)}(B_0(S))\geq b+1$.
 \end{proof}

 \begin{lem}\label{lem:boundclassicalt}
 With the above notation, we have at least $w$ unipotent characters in $B_0(S)$ that are not $\aut(S)$-conjugate, and hence $w\leq k_{\Aut(S)}(B_0(S))$. If $t\geq 1$, this yields $p\leq p^t\leq k_{\Aut(S)}(B_0(S))$. 
 \end{lem}
 \begin{proof}
We remark first that the unipotent characters of $H$ are irreducible on restriction to $\Omega$ and are trivial on $\zent{H}$.  (See, e.g. \cite[Proposition 2.3.15]{GM20}.)

In the case $H=\GL_n^\epsilon(q)$, we have the number of unipotent characters in $B_0(H)$ is $k(e,w)$, by \cite[Proposition 
(2.3)]{MO83}, where $k(e,w)$ can be computed as in \cite[Lemma 1]{olsson84}. This yields at least $k(e,w) \geq w$ unipotent characters in $B_0(S)$,  which are all $\aut(S)$-invariant (see \cite[Theorem 2.5]{Malle08}). 

If $H=\Sp_{2n}(q)$ or $\SO_{2n+1}(q)$, we see from \cite[Section 5.2]{malle17} that the number of unipotent characters in $B_0(H)$ is $k(2e, w)>2w$, which again are $\aut(S)$-invariant by \cite[Theorem 2.5]{Malle08} unless $H=\Sp_{4}(q)$ with $q$ even. In the latter case, the unipotent characters are at worst permuted in pairs by $\aut(S)$, and hence again there are at least $w$ non-$\aut(S)$-conjugate such characters.

If $H=\SO_{2n}^\pm(q)$, then $B_0(H)$ contains either at least $k(2e,w)$ unipotent characters or at least $(k(2e,w)+3k(e,w/2))/2$ when $w$ is even, using \cite[Section 5.3 and Lemma 5.6]{malle17}. One can see that these numbers are again at least $2w$, and by \cite[Theorem 2.5]{Malle08}, again the unipotent characters are at worst permuted in pairs by $\aut(S)$ unless $H=\SO_8^+(q)$. In the latter case, 
\cite[Lemma 3.10]{RSV21} gives the claim.
 \end{proof}

\subsection{The Proof of Theorem \ref{thm:BP21almostsimple}}

The following will be useful in the proof of Theorem  \ref{thm:BP21almostsimple}, as well as in the proof of Theorem \ref{thm:BP21principal} below.
Here for a group $G$, we write $k_p(G)$ to denote the number of conjugacy classes of $p$-elements of $G$.

\begin{lem}
\label{kp}
Let $G$ be a finite group. Then $k_p(G)\leq k(B_0(G))$.  In particular, the number of chief factors of $G$ of order divisible by $p$ is at most $k(B_0(G))$. 
\end{lem}

\begin{proof}
Let $\{x_1,\dots,x_t\}$ be a set of representatives of the non-central conjugacy classes of $p$-elements of $G$. 
By \cite[Theorem 4.14]{nbook}, $B_0(\bC_G(x_i))^G$ is defined for every $i=1,\dots,t$. 
By  \cite[Theorem 5.12]{nbook}) and Brauer's third main theorem (\cite[Theorem 6.7]{nbook}), we have that 
$$
k(B_0(G))=l(B_0(G))|\bZ(G)|_p+\sum_{i=1}^t l(B_0(\bC_G(x_i)))\geq|\bZ(G)|_p+t=k_p(G),
$$
as wanted.
\end{proof}
Finally, we can prove Theorem \ref{thm:BP21almostsimple}.

\begin{proof}[Proof of Theorem \ref{thm:BP21almostsimple}]

Recall from Remark \ref{rem:k7} that we may assume that $k:=k(B_0(A))\geq 7$. If $S$ is a sporadic group, Tits group,  group of Lie type with exceptional Schur multiplier, or alternating group $\alt_n$ with $n\leq 7$, then the result is readily checked using GAP and its Character Table Library \cite{GAP}. 
 We therefore assume that $S$ is not one of these groups. Throughout, let $P\in\Syl_p(A)$ and $P_0\in\Syl_p(S)$ such that $P_0=P\cap S$.

(I) If $S$ is an alternating group with $n\geq 8$, then  $A\in\{\alt_n, \sym_n\}$. Note that $2k\geq k(B_0(\sym_n))$. Let $n=pw+r$ with $0\leq r<p$. Then we have  $|P|\leq (2,p)\cdot|P_0|=(n!)_p=((pw)!)_p\leq p^{pw}$ by \eqref{eq:sizeSylowSw}. Further, by \cite[Theorem 1.10]{MO83} and \cite[Lemma 1 and p. 44]{olsson84}, we have 
\[k(B_0(\sym_{n}))=k(B_0(\sym_{pw}))=k(p,w)
>\pi(w)p\geq wp,\] where $\pi(w)$ denotes the number of partitions of $w$. Then \[|P_0|\leq |P|\leq p^{pw}<p^{2k}\leq \left(\frac{k^4}{4}\right)^{k}\] when combined with Lemma \ref{lem:HSF}, yielding a bound stronger than (a)-(c) in this case. 

\medskip

From now on, we assume $S$ is a simple group of Lie type. Let $S=G/\zent{G}$, where $G=\bG^F$ for a simple, simply connected reductive group $\bG$ and a Steinberg endomorphism $F\colon \bG\rightarrow\bG$, where $\zent{G}$ is the full, nonexceptional Schur covering group of $S$.  Write $k_0:=k_{\aut(S)}(B_0(S))$ so that $k_0\leq k$.

(II) We will first show \ref{thm:BP21simple}.  

First, assume $S$ is defined in characteristic $p$, so that $|P_0|=q^{|\Phi^+|}$, where $\Phi^+$ is the set of positive roots of $\bG$  (see \cite[Proposition 24.3]{MT11}).  
 We have $\irr{B_0(S)}=\irr{S}\setminus\{\mathrm{St}_S\}$, where $\mathrm{St}_S$ is the Steinberg character (see \cite[Theorem 6.18]{CE04}), and $k_0\geq \frac{q^r}{|\zent{G}|\cdot|\mathrm{Out}(S)|}$, as in \cite[Section 2D]{HSF21}.  
 Let $f$ be the integer (or half-integer, in the case of Suzuki and Ree groups $\tw{2}\type{G}_2(q^2)$, $\tw{2}\type{F}_4(q^2)$, $\tw{2}\type{B}_2(q^2)$) such that $q=p^f$, and note that $\sqrt{q^r}=p^{rf/2}\leq p^{rf}/f=q^r/f$, unless $q=8$ and $r=1$. In the latter case, $S=\PSL_2(8)$ and $|P_0|=8$, so the statement  holds. So, we assume $(q,r)\neq (8,1)$.

 Here, we include the full argument for the groups $\PSL_n^\epsilon(q)$ ($n\geq 2$), which correspond to $\bG$ of type $\type{A}_{n-1}$. Table \ref{tab:defchar} gives relevant values for various groups of Lie type, and from this information, the arguments in the other cases are similar.  So, let $S=\PSL_n^\epsilon(q)$. Then $|P_0|=q^{n(n-1)/2}$; $r=n-1$; $|\out(S)|\leq 2f\cdot (n, q-\epsilon)\leq 2fn$; and $|\zent{G}|=(n, q-\epsilon)\leq n$.   By Corollary \ref{cor:nunips}, we have  $k_0\geq n$. Together, this gives 
\[k_0\geq \frac{q^{n-1}}{2n^2f}\geq \frac{q^{(n-1)/2}}{2n^2}\geq \frac{q^{(n-1)/2}}{2k_0^2}.\] 
Then $q^{(n-1)/2}\leq 2k_0^3\leq k_0^4$, so $|P_0|=q^{n(n-1)/2}\leq k_0^{4n}\leq k_0^{4k_0}\leq k_0^{2k_0^2}$.

Finally, we may assume $S$ is a group of Lie type defined in characteristic different than $p$. If $S$ is of exceptional type, then Propositions \ref{prop:BP21except2} and \ref{prop:BP21except} yield \ref{thm:BP21simple}.

Hence, we may assume $S$ is of classical type, and we let $H, \wt{P}$, and $ \bar{P}$ be as in Section \ref{sec:classical}. Recall that we have $|P_0|\leq |\wt{P}|$. If $p=2$, we further have $|\wt{P}|\leq |\GL_n(q^2)|_2\leq 2^{(b+1)n}(n!)_2\leq 2^{(b+2)n}$, where  $2^{b+1}$ is the largest power of $2$ dividing $q^2-1$ and the last inequality is from \eqref{eq:sizeSylowSw}. In particular using Lemma \ref{lem:boundclassicalb2} and Corollary \ref{cor:nunips}, in this case $|P_0|\leq 2^{k_0^2+k_0}<k_0^{2k_0^2}.$

Now we assume $p$ is odd. If $\wt{P}$ is abelian, note that $\wt{P}=\bar{P}$ in the notation before. Then $|P_0|\leq p^{bw}<k^{2k_0^2}$, from Lemmas \ref{lem:HSF}, \ref{lem:boundclassicalb}, and \ref{lem:boundclassicalt}, and the statement holds. 
We are left with the case that $S$ is classical and $t\geq 1$. Then  by Lemmas \ref{lem:boundclassicalb} and \ref{lem:boundclassicalt}, along with \eqref{eq:sizeSylowSw}, we see that \[|P_0|\leq p^{bw}\cdot (w!)_p\leq p^{bw}\cdot p^{w}=p^{(b+1)w}\leq k_0^{k_0^2},\] which completes the proof of \ref{thm:BP21simple}.

(III) We now complete the proof of \ref{thm:BP21simplemain}.
 Let $G$  be defined over $\mathbb{F}_q$, where $q=q_0^f$ for some prime $q_0$ and integer $f$. (In the case of Suzuki and Ree groups, we instead let $q^2:=q_0^f$ with $f$ an odd integer.) Further, write $f:=p^{f'}\cdot m$ with $(m,p)=1$. 
 
 From part \ref{thm:BP21simple}, recall that $|P_0|\leq k^{2k^2}$. 
Note that $|P/P_0|=|A/S|_p$ and this number is at most $p^{f'+1}$ unless $S=\type{D}_n(q)$ or $\tw{2}\type{D}_n(q)$ with $p=2$ and $|A/S|_2\leq 2^{f'+3}$ or $S=\PSL_n^\epsilon(q)$ with $n\geq 3$ and $p\mid(n,q-\epsilon)$, in which case  $|A/S|_p$ divides $2p^{b+f'}$ with $p^b\mid\mid (q-\epsilon)$.

Recall that $\aut(S)=\wt{S}\rtimes \mathcal{D}$ with $\mathcal{D}$ a group of field and graph automorphisms as before. A Sylow $p$-subgroup of   $\wt{S}A\cap \mathcal{D}$ contains a cyclic group of size $p^{f''}$, where  $f''\leq f'$ and $|\wt{S}A\cap\mathcal{D}|_p\leq p^{f''+1}$. Then $A$ must also contain an element of order $p^{f''}$,  and hence elements of  orders $p^i$ for $1\leq i\leq f''$. Then $k_p(A)\geq f''$, and hence $k\geq f''$ by Lemma \ref{kp}. 

Now, if $S$ is not one of the exceptions mentioned above, we have  $|A|_p\leq |\wt{S}A|_p=|\wt{S}|_p\cdot |\wt{S}A\cap\mathcal{D}|_p\leq |P_0|\cdot p^{f''+1}$. If  $S=\type{D}_n(q)$ or $\tw{2}\type{D}_n(q)$ with $p=2$, we have $|\wt{S}A|_2\leq |P_0|\cdot 2^{f''+3}$. If $S=\PSL_n^\epsilon(q)$ with $n\geq 3$ and $p\mid(n,q-\epsilon)$, we have  $|\wt{S}A|_p\leq |P_0|\cdot  p^{b+f''+1}$, where $p^b\mid\mid(q-\epsilon)$. Then using \ref{thm:BP21simple} and Lemmas \ref{lem:boundclassicalb2}, and \ref{lem:boundclassicalb}, we have in each case that
\[|P|=|A|_p\leq k^{2k^2}\cdot p^{f''+k}.\] 
 Combining the above with Lemma \ref{lem:HSF}, we obtain $|P|\leq k^{2k^2}\cdot p^{2k}< k^{2k^2}\cdot k^{4k}$, completing the proof. 
\end{proof}

\begin{table}[H]\label{tab:defchar}\caption{Relevant Data for Bounding $|P_0|$ in Defining Characteristic}
\begin{tabular}{|c|c|c|c|c|}
\hline
Type of $\bG$ & Size of $\Phi^+$ & Rank $r$ & upper bound for $|\zent{G}|$ & upper bound for $|\out(S)|$\\
\hline\hline
$\type{A}_{n-1}$, $n\geq 2$ & $n(n-1)/2$ & $n-1$ & $n$ & $2fn$\\
\hline
$\type{B}_n$ or $\type{C}_n$, $n\geq 3$ & $n^2$ & $n$ & 2 & $2f$ \\
\hline
$\type{B}_2$ & $4$ & $2$ & $2$ & $2f$\\
\hline
$\type{D}_n$, $n\geq 5$ & $n(n-1)$ & $n$ & 4 & $8f$ \\
\hline

$\type{D}_4$ & $12$ & $4$ & 4 & $24f$ \\
\hline
$\type{F}_4$ & $24$ & $4$ & 1 & $2f$ \\
\hline
$\type{G}_2$ & $6$ & $2$ & 1 & $2f$ \\
\hline
$\type{E}_6$ & $36$ & $6$ & 3 & $6f$ \\
\hline
$\type{E}_7$ & $63$ & $7$ & 2 & $2f$ \\
\hline
$\type{E}_8$ & $120$ & $8$ & 1 & $f$ \\
\hline

\end{tabular}
\end{table}

\section{Proof of Theorem A}\label{sec:BP21proof}

In this section we complete the proof of Theorem \ref{thm:BP21principal}. We begin with some additional general observations that will be useful in  the proof.

\begin{lem}\label{lem: boundingcovering} Let $G$ be a finite group and let $N\trianglelefteq G$. If $b\in\Bl(N)$ is covered by $B\in\Bl(G)$ then $k(b)\leq|G:N|k(B)$.
\end{lem}
\begin{proof}
This is a direct consequence of \cite[Theorem 9.4]{nbook}
\end{proof}

\begin{lem}\label{lem:pelements} Let $G$ be a finite group and suppose that 
$N=S_1\times\cdots\times S_n$ is a normal subgroup, where 
where $S_i$ is simple nonabelian  and $p$ divides $|S_i|$ for all $i$. Then $n\leq k(B_0(G))$.
\end{lem}
\begin{proof}
Let $1\neq x_i\in S_i$ be a $p$-element for every $i$. Note that $G$ acts on $\{S_1,\dots, S_n\}$ by conjugation. Therefore,  the elements $$(x_1,1,\dots,1), (x_1,x_2,1,\dots,1),\dots,(x_1,\dots,x_n)$$ are representatives of $n$ different conjugacy classes of $p$-elements of $G$. By Lemma \ref{kp}, $n\leq k(B_0(G))$. 
\end{proof}

\begin{lem}
\label{proj}
Suppose that $S_1,\dots, S_n$ are nonabelian simple groups of order divisible by a prime number $p$ and let $S_1\times\cdots\times S_n\leq G\leq\Aut(S_1)\times\cdots\times\Aut(S_n)$.  Let $k=k(B_0(G))$, where $B_0(G)$ is the principal $p$-block of $G$.
Then $$|G|_p\leq k^{4k^3}.$$
\end{lem}

\begin{proof}

Write $A_i=\Aut(S_i)$ and $A=A_1\times\cdots\times A_n$. Let $\pi_i$ be the restriction to $G$ of the projection  from $A$ onto $A_i$ for every $i$. Set $K_i=\Ker\pi_i$.  Notice  that  $G/K_i$ is isomorphic to an almost simple group $G_i$ with socle $S_i$.  Furthermore, the intersection of the $K_i's$ is trivial, so $G$ embeds into the direct product of the groups $G/K_i$.  Furthermore, $B_0(G/K_i)\subseteq B_0(G)$ for every $i$. 

By Theorem \ref{thm:BP21almostsimple}, we have that 
$$|G/K_i|_p\leq k^{4k^2} $$
for every $i$. 
Since $S_i$ is normal in $G$ for all $i$, by Lemma \ref{lem:pelements} we have that $n\leq k$, and hence 
$$|G|_p\leq\prod_{i=1}^n |G/K_i|_p\leq   k^{4k^3},$$
as desired.\end{proof}

We define the function 
$$f(k)=k^{k}\cdot (k!k)^{4(k!k)^3}.$$

\begin{thm}
\label{rad}
Let $G$ be a finite group and let $R$ be the $p$-solvable radical of $G$. Then $|G:R|_p\leq f(k)$, where $k=k(B_0(G))$.  
\end{thm}

\begin{proof}
Without loss of generality, we may assume that $R=1$. Let $F=\bF^*(G)$ be the generalized Fitting subgroup, which in this case is a direct product of non-abelian simple groups of order divisible by $p$. Write $F=S_1\times\cdots\times S_n$. By Lemma \ref{lem:pelements}, we obtain that $n\leq k$. Since $\bC_G(F)\leq\bZ(F)=1$, it follows that $G$ embeds into $\Gamma=\Aut(F)$. Note that $A=\Aut(S_1)\times\cdots\times\Aut(S_n)$ is a normal subgroup of $\Gamma$ and $\Gamma/A$ is isomorphic to a subgroup of $\SSS_n$. In particular, $$|\Gamma/A|\leq n!\leq k!.$$ 

Put $N=G\cap A$ and note that $|G:N|\leq k!$. By the well-known Legendre's inequality, we have that $(k!)_p\leq p^k$, so  $(k!)_p\leq k^k$. Write $k'=k(B_0(N))$. It follows from Lemma \ref{proj} that
$$
|G|_p=|G:N|_p|N|_p\leq k^{k}\cdot  k'^{4k'^3}.
$$

Now, Lemma \ref{lem: boundingcovering} implies that 
$$
|G|_p\leq k^{k}\cdot (k!k)^{4(k!k)^3},
$$
as wanted.
\end{proof}

Recall that the socle $\Soc(G)$ of a finite group $G$ is the product of the minimal normal subgroups of $G$. We can write $\Soc(G)=A(G)\times T(G)$, where $A(G)$ is the product of the abelian minimal normal subgroups of $G$ and $T(G)$ is the product of the non-abelian minimal normal subgroups of $G$. Note that $T(G)$ is a direct product of non-abelian simple groups.

Finally, we set $g(k)=2^{2^k}f(k)^k$. 
The following completes the proof of Theorem \ref{thm:BP21principal}.

\begin{thm}
\label{exp}
Let $G$ be a finite group. Let $k=k(B_0(G))$. Then $|G|_p\leq g(k)$.
\end{thm}

\begin{proof}
Let $O_1=\bO_{p'}(G)$. Set $E_1/O_1=T(G/O_1)$, where $T(G/O_1)$ is the non-abelian part of the socle of $G/O_1$. For $i>1$, we define $O_i/E_{i-1}=\bO_{p'}(G/E_{i-1})$ and $E_i/O_i=T(G/O_i)$, so that we have a normal series $1\leq O_1\leq E_1\leq O_2\leq E_2\leq\cdots$. Note that if $O_i<E_i$ then $E_i/O_i$ is a direct product of simple groups of order divisible by $p$. By Lemma \ref{kp} we conclude that 
$O_{k+1}=E_{k+1}=O_{k+2}=\cdots$. Set $O=O_{k+1}$. Note that $F/O=\bF^*(G/O)$ is a $p$-group. Since $\bC_{G/O}(F/O)\leq F/O$, \cite[Corollary V.3.11]{fei} implies that $B_0(G/O)$ is the unique $p$-block of $G/O$. Since $B_0(G/O)\subseteq B_0(G)$, Landau's theorem implies that 
$$|G/O|\leq2^{2^k}.$$

Now, for $i\leq k$, let $C_i/O_i=\bC_{G/O_i}(E_i/O_i)$, so that $G/C_i$ is isomorphic to a subgroup of $\Aut(E_i/O_i)$ that contains $E_iC_i/C_i\cong E_i/O_i$. Notice that the $p$-solvable radical of $G/C_i$ is trivial, so by Theorem  \ref{rad} applied to $G/C_i$, we have that 
$$
|E_i/O_i|_p\leq|G/C_i|_p\leq f(k).
$$
It follows that 
$$
|G|_p=|G:O|_p\prod_{i=1}^k|E_i/O_i|_p\leq 2^{2^k}f(k)^k,
$$
as wanted.
\end{proof}

\begin{rmk} {\rm
Arguing in a similar way, we can see that if a finite group $G$ does not have simple groups of Lie type in characteristic different from $p$ as composition factors, then $|G:\bO_{p'}(G)|$ can be bounded above in terms of $k(B_0(G))$. We sketch the proof. First, we know that if $p$ is a prime and $S$ is a simple group of Lie type in characteristic $p$ or an alternating group, then $|S|$ is bounded from above in terms of $|S|_p$. Therefore, the same happens for almost simple groups with socle of Lie type in characteristic $p$ or alternating. Now, let $R$ be the $p$-solvable radical of a finite group $G$. We can argue as in the proof of Theorem \ref{rad} to see that $|G:R|$ is bounded from above in terms of $k(B_0(G))$. Using Lemma \ref{lem: boundingcovering}, we see that $k(B_0(R))$ is bounded from above in terms of $k(B_0(G))$. Since $R$ is $p$-solvable, $\Irr(B_0(R))=\Irr(R/\bO_{p'}(R))$. Using that $\bO_{p'}(R)=\bO_{p'}(G)$ and Landau's theorem, we deduce that $|R:\bO_{p'}(G)|$ is bounded from above in terms of $k(B_0(R))$. The result follows.}
\end{rmk}

\begin{rmk}{\rm 
We have already mentioned that the case $p=2$ of Brauer's Problem 21 was already known by 
\cite{Kulshammer-Robinson} and \cite{ruh}. However, this relies on Zelmanov solution of the restricted Burnside problem. As discussed in \cite{VZ99} the bounds that are attainable in this problem are of a magnitude that is incomprehensibly large. The bound that we have obtained for principal blocks, although surely far from best possible, is much better than any bound that relies on the restricted Burnside problem.}
\end{rmk}

Recently, there has been a large interest in studying relations among (principal) blocks for different primes. For instance, what can we say about the set of irreducible characters that belong to some principal block? The groups with the property that all irreducible characters belong to some principal block were determined in \cite{BZ11}. 
As a consequence of Brauer's Problem 21 for principal blocks, we see that for any integer $k$ there are finitely many groups with at most $k$ irreducible characters in some principal block. Note that this is a strong form of Landau's theorem. In this corollary, given a prime $p$,  we write $B_p(G)$ to denote the principal $p$-block of $G$.

\begin{cor}
The order of a finite group is bounded from above in terms of $|\bigcup_p\Irr(B_p(G))|$.
\end{cor}

\begin{proof}
By Theorem A, we know that for any prime $p$, $|G|_p$ is bounded from above in terms of $k(B_p(G))$. It follows that 
$|G|_p$ is bounded from above in terms of $|\bigcup_p\Irr(B_p(G))|$. In particular, if $p$ is a prime divisor of $|G|$, then $p$ is bounded from above in terms of $|\bigcup_p\Irr(B_p(G))|$. The result follows.
\end{proof}

\section{Blocks with three irreducible characters}\label{sec:kb3reduction}

In this section, we prove Theorem C. As usual, if $B$ is a $p$-block of a finite group $G$, $l(B)$ is the number of irreducible $p$-Brauer characters in $B$. 
By \cite{K84}, we know that if $k(B)=3$ and  $l(B)=1$, then the defect group is cyclic of order $3$. So we are left with the case $l(B)=2$.

\begin{lem}\label{lemamoritanilpotent} Let $N\lhd G$ and let $B$ be a $p$-block of $G$ with defect group $D$. Suppose that $B$ covers a $G$-invariant block $b$ of $N$ such that $D$ is a defect group of $b$. If $b$ is nilpotent then $k(B)=k(B')$, where $B'\in{\rm Bl}(\norm GD)$ is the Brauer first main correspondent of $B$.
\end{lem}

\begin{proof}
Since $D$ is a defect group of $b$, we have that the Harris-Kn\"orr correspondent of $B$ (see \cite[Theorem 9.28]{nbook}) with respect to $b$ is $B'$, the Brauer first main correspondent of $B$. By the work in \cite{KP90} (see the explanation at the beginning of \cite[Section 3]{Rob02}, for instance) we have that $B$ and $B'$ are Morita equivalent, and hence, they have the same number of irreducible characters.
\end{proof}

We write ${\rm cd}(B)$ to denote the set of degrees of the irreducible (ordinary) characters in $B$. We write $k_0(B)$ to denote the number of irreducible (ordinary) characters of height zero in $B$.
The following is Theorem \ref{thm:kb3reduction}.

\begin{thm}  Let $G$ be a finite group and let $B$ be a $p$-block of $G$.  Suppose that Condition B holds for $(S,p)$ for all simple non-abelian composition factors $S$ of $G$. Let $D$ be a defect group of $B$. If $k(B)=3$, then $|D|=3$.
\end{thm}

\begin{proof}

We proceed by induction on $|G|$. Notice that we may assume $D>1$ is elementary abelian by Brauer theorem (see \cite[Theorem 3.18]{nbook}, for instance) and \cite[Corollary 7.2]{hkks} and that $l(B)=2$ by \cite{K84}.  

\medskip

\textit{Step 0. We may assume that $p$ is odd.} 

Suppose that $p=2$. By \cite[Corollary 1.3(i)]{Lan81}, if $|D|>2$ we have that $4$ divides $k_0(B)\leq k(B)=3$, which is absurd. Hence we have that $|D|=2$. But in this case we know that $k(B)=2$, by \cite{Br82}. This is a contradiction, so $p$ is odd.

\medskip

\textit{Step 1. We may assume $\bO_{p}(G)=1$.} 

Let $M={\rm \textbf{O}}_{p}(G)$, and let $\bar{B}\in {\rm Bl}(G/M)$ dominated by $B$ with defect group $D/M$ (\cite[Theorem 9.9(b)]{nbook}). Since $M$ is a $p$-group, it has just one $p$-block, the principal one, so $B$ covers $B_0(M)$. By \cite[Theorem 9.4]{nbook} if $1_M\neq \theta\in{\rm Irr}(M)$, there is $\chi\in{\rm Irr}(B)$ over $\theta$, hence $\chi$ does not lie in $\bar{B}$. Then $k(\bar{B})\leq 2$. If $k(\bar{B})=1$, then $D=M$ and $D$ is normal in $G$. In this case $|D|=3$ by \cite[Theorem 4.1]{KNST14}. If $k(\bar{B})=2$, then by \cite{Br82}, we have $p=2$. This contradicts Step 0.

\medskip

\textit{Step 2. If $N$ is a normal subgroup of $G$, and $b$ is a $p$-block of $N$ covered by $B$, we may assume that $b$ is $G$-invariant.} 

Let $G_b$ be the stabilizer of $b$ in $G$. By the Fong-Reynolds correspondence (\cite[Theorem 9.14]{nbook}), if $c$ is the block of $G_b$ covering $b$ such that $c^G=B$, we have that $k(c)=k(B)=3$ and if $E$ is a defect group of $c$, then $E$ is a defect group of $B$. If $G_b<G$, by induction we are done.

\medskip 

\textit{Step 3. We may assume that if $N$ is a normal subgroup of $G$ and $b$ is a $p$-block of defect zero of $N$ covered by $B$, then $N$ is central and cyclic. In particular, we may assume that $\zent G=\oh {p'} G$ is cyclic.} 

Write $b=\{\theta\}$. Since $\theta$ is of defect zero, we have that $(G,N,\theta)$ is an ordinary-modular character triple and there exists $(G^*,N^*,\theta^*)$ an isomorphic ordinary-modular character triple with $N^*$ a $p'$-group central in $G^*$ and cyclic (see \cite[Problems 8.10 and 8.13]{nbook}). Notice also that since $G^*/N^*\cong G/N$, the set of non-abelian composition factors of $G^*$ is contained in the set of non-abelian composition factors of $G$, so Condition B holds for all non-abelian composition factors of $G^*$. 
 If $$*:{\rm Irr}(G|\theta)\rightarrow {\rm Irr}(G^*|\theta^*)$$ is the bijection given by the isomorphism of character triples and $B=\{\chi_1,\chi_2,\chi_3\}$, we have that $B^*=\{\chi_1^*,\chi_2^*,\chi_3^*\}$ is a $p$-block of $G^*$. Now, if $D^*$ is a defect group of $B^*$ and $|D^*|=3$, we claim that $|D|=3$ (notice that in this case $p=3$, so we just need to prove that $|D|=p$). Indeed, let $\chi\in{\rm Irr}(B)$ of height zero. Since isomorphism of character triples preserves ratios of character degrees and all the characters in $B^*$ are of height zero (because $D^*$ has prime order), we have

$$\frac{|G:D|_p}{|N|_p}=\left(\frac{\chi(1)}{\theta(1)}\right)_p=\chi^*(1)_p=|G^*:D^*|_p=\frac{|G^*:N^*|_p}{p}=\frac{|G:N|_p}{p}.$$ 
Since $b$ is $G$-invariant, we have that $D\cap N$ is a defect group of $b$ by \cite[Theorem 9.26]{nbook}, so $D\cap N=1$ because $b$ has defect zero. Now,  we have
$$|D|=\frac{|G|_p}{|G:D|_p}=\frac{|G:N|_p|N|_p}{|G:D|_p}=p,$$ as claimed. Hence we may assume that $N$ is central and cyclic. In particular, by Step 1 we have that $\zent G=\oh {p'} G$ is cyclic.

\medskip

\textit{Step 4. There is a unique $G$-conjugacy class of non-trivial elements in $D$.}

Let $b$ be the $p$-block of $N$ covered by $B$. Since $b$ is $G$-invariant, we have that $D\cap N$ is a defect group of $b$ by \cite[Theorem 9.26]{nbook}. By a theorem of Brauer (\cite[Theorem 5.12]{nbook}) we have that $$k(B)=l(B)|\zent G|_p+\sum_{i=1}^k\sum_{\substack{b\in{\rm Bl}(\cent G {x_i})\\ b^G=B}} l(b),$$ where $\{x_1,x_2,\ldots,x_k\}$ are the representatives of the non-central $G$-conjugacy classes of $p$-elements of $G$. Since $|\zent G|_p=1$ by Step 1, we have 

$$k(B)=l(B)+\sum_{i=1}^k\sum_{\substack{b\in{\rm Bl}(\cent G {x_i})\\ b^G=B}} l(b).$$

By \cite[Theorem 4.14]{nbook}, if $x_i\in D$, then there is $b\in{\rm Bl}(\cent G {x_i})$ such that $b^G=B$. Since $l(B)=2$ and $k(B)=3$, we have that there is just one $G$-conjugacy class of non-trivial elements in $D$. 

\medskip

\textit{Step 5. If $N$ is a non-central normal subgroup of $G$, then $D\leq N$. In particular, if $b$ is the only block of $N$ covered by $B$, then $D$ is a defect group of $b$.} 

Let $b$ be the $p$-block of $N$ covered by $B$. Again, since $b$ is $G$-invariant, we have that $D\cap N$ is a defect group of $b$ (by \cite[Theorem 9.26]{nbook}). Since $D\cap N>1$ (otherwise $b$ is of defect zero and $N$ is central by Step 3), we have that there is an element $1\neq x\in D\cap N$. If $1\neq y\in D$, $y$ is $G$-conjugate to $x$ by Step 4 and thus $y\in N$, as wanted.

\medskip

\textit{Step 6. If $N$ is a normal subgroup of $G$, $b$ is the block of $N$ covered by $B$ and all the irreducible characters in $b$ have the same degree, then $N$ is central.}

Suppose that $N$ is not central. By Step 5 we have that $D$ is a defect group of $b$.  By \cite[Proposition 1 and Theorem 3]{OT83} we have that $D$ is abelian and has inertial index 1. By \cite[1.ex.3]{BP80}, we know that $b$ is nilpotent. Hence by Lemma \ref{lemamoritanilpotent} we have that $k(B')=k(B)=3$, where $B'$ is the Brauer first main correspondent of $B$ in $\norm G D$. If $\norm G D< G$, by induction we are done. Hence we may assume that $D\lhd G$, but this is a contradiction with Step 1. Therefore $N$ is central.

\medskip

\textit{Step 7. If $N$ is a normal subgroup of $G$, and $b$ is the unique block of $N$ covered by $B$, then ${\rm Irr}(b)$ has at most three $G/\cent G N$-orbits}. 

Suppose that there are more than three $G/\cent G N$-orbits in ${\rm Irr}(b)$, and let $\theta_i\in{\rm Irr}(b)$ be a representative for these orbits (so there are at least four of them). By \cite[Theorem 9.4]{nbook} we can take $\chi_i\in{\rm Irr}(B)$ lying over $\theta_i$. By Clifford's theorem, the $\chi_i$ are all different. But this is a contradiction since $k(B)=3$.

\medskip

\textit{Step 8. We may assume that $D$ is not cyclic}. 

Otherwise, by Dade's theory of blocks with cyclic defect \cite{Dad66}, we have that $k(B)=k_0(B)=k_0(B')=k(B')$ where $B'\in{\rm Bl}(\norm G D|D)$ is the Brauer correspondent of $B$, and hence we may assume that $D$ is normal in $G$. In this case we are done by Step 1.

\medskip

\textit{Step 9. Write $Z=\zent G$ and $\overline{G}=G/Z$. Then $\overline{G}$ has a unique minimal normal subgroup $\overline{K}=K/Z$, which is simple.}

Let $K/Z$ be a minimal normal subgroup of $G/Z$. Since $Z={\rm \textbf{O}}_{p'}(G)$, we have that $K/Z$ is not a $p'$-group. Since ${\rm \textbf{O}}_p(G)=1$, $K/Z$ is not a $p$-group. Hence $K/Z$ is semisimple. Notice that $K/Z$ is the unique minimal normal subgroup of $G/Z$. Indeed, if $K_1/Z , K_2/Z $ are minimal normal subgroups of $G/Z$, then by Step 5, $D\subseteq K_1\cap K_2=Z =\oh {p'} G$ and hence $D=1$, a contradiction. 

\smallskip

Write $\overline{K}=K/Z$. Then $\overline{K}=\overline{S_1}\times\cdots\times\overline{S_t}$, where $\overline{S_i}$ is non-abelian simple and $\overline{S_i}=\overline{S_1}^{g_i}$ for some $g_i\in G$. Write $\overline{S_i}=S_i/Z$ and notice that $S_i=S_1^{g_i}$. Notice that since $S_i/\zent {S_i}=S_i/Z$ is simple, we have that $S_i'$ is a component of $G$ and hence $[S_i',S_j']=1$ whenever $i\neq j$ (\cite[Theorem 9.4]{isaacs}). Furthermore, $S_i=S_i'Z$, so $[S_i,S_j]=1$ whenever $i\neq j$. 

\smallskip

We want to show that $t=1$. By Step 5 we have that $D$ is a defect group of $b$, the only block in $K$ covered by $B$. If $D\cap S_i=1$ for all $i=1,\ldots,t$, then $D=1$, a contradiction. Hence there is $i$ such that $D\cap S_i>1$. Without loss of generality we may assume that $D\cap S_1>1$. Let $b_1$ be the only block of $S_1$ covered by $b$ and notice that, since $b_1$ is $K$-invariant, $D\cap S_1$ is a defect group of $b_1$.  We claim that $D\not\subseteq S_1$.  Suppose otherwise. Notice that $D^{g_i}$ is a defect group of $b^{g_i}=b$ and hence $D^{g_i}=D^k$ for some $k\in K$.  Now, $D^{g_i}=D^k\subseteq S_1^{g_i}\cap S_1^k=S_i\cap S_1=Z$, which is a $p'$-group. This is a contradiction, so $D\not\subseteq S_1$.  

\smallskip 

Let $1\neq x\in D\cap S_1$. If $D\cap S_i=1$ for all $i\neq 1$, we have that $D=D\cap S_1$ which is a contradiction by the previous paragraph. Hence there is $i\neq 1$ such that $D\cap S_i\neq 1$. Let $1\neq x_i\in D\cap S_i$. Now $xx_i,x\in D$ and by Step 4 we have that $x$ and $xx_i$ are $G$-conjugate, which is not possible. Hence $t=1$, as wanted.


\medskip

\textit{Step 10. Final step.}

Now $K'$ is a quasi simple group with center a cyclic $p'$-group. If $b$ is the unique block of $K'$ covered by $B$, we have that $D$ is a defect group of $b$ by Step 5 and hence is not cyclic elementary abelian by Step 8. We claim that $b$ is faithful. Let $X={\rm ker}(b)$. By Theorem \cite[Theorem 6.10]{nbook}, we have that $X\leq Z\cap K'$. Now, let $\psi\in{\rm Irr}(b)$, then $\psi$ lies over $1_X$ and hence, there is $\chi\in{\rm Irr}(B)$ lying over $1_X$. Now, by \cite[Theorem 9.9 (c)]{nbook} we have that $k(\bar{B})=k(B)=3$, where $\bar{B}$ is the block of $G/X$ containing $\chi$. If $X>1$, by induction we obtain that $|D|=|DX/X|=3$, and we are done. Hence we may assume that $X=1$. By Condition B, there are at least four ${\rm Aut}(K')$-conjugacy classes of irreducible characters in $b$, which is a contradiction by Step 7.
\end{proof}

\section{On Condition \ref{quasisimples2}}\label{sec:orbits}

We end the paper with a discussion on Condition \ref{quasisimples2}. In \cite[Theorem B]{MNST}, a statement similar to Condition \ref{quasisimples2} but requiring only 3 distinct orbits is proven. Unfortunately, for groups of Lie type in non-defining characteristic, the strategy used there is not quite sufficient to obtain 4 orbits. In fact, we will see that this is not always attainable. However, here we address  several situations in which  we do obtain Condition \ref{quasisimples2}.

\begin{prop}\label{prop:sporadic}
Let $p\geq 3$ be prime. Let $K$ be a quasisimple group with $\zent{K}$ a cyclic $p'$-group and socle $K/\zent{K}$  a simple sporadic group, the Tits group $\tw{2}\type{F}_4(2)'$, $\type{G}_2(2)'$, $ \tw{2}\type{G}_2(3)'=\type{A}_1(8)$,  a simple group of Lie type with exceptional Schur multiplier, or an alternating group $\alt_n$ with $5\leq n\leq 13$.  Let $B$ be a $p$-block for $K$ with noncyclic, positive defect.  Then $|\mathrm{cd}(B)|\geq 4$, with the following exceptions when $p=3$:
\begin{itemize}
\item $K=2.\alt_7$; $B$ is Block 3 in GAP;  $|\mathrm{cd}(B)|= 3$; and $k_{\aut(K)}(B)= 4$
\item $K=2.\alt_8$; $B$ is Block 5 in GAP;  $|\mathrm{cd}(B)|= 3$; and $k_{\aut(K)}(B)= 4$
\item $K=2.\alt_{11}$; $B$ is Block 5 in GAP;  $|\mathrm{cd}(B)|= 3$; and $k_{\aut(K)}(B)= 4$
\item $K=2.\alt_{13}$; $B$ is Block 5 in GAP;  $|\mathrm{cd}(B)|= 3$; and $k_{\aut(K)}(B)= 4$
\item $K=\tw{2}\type{G}_2(3)'=\type{A}_1(8)$; $B=B_0(K)$; $|\mathrm{cd}(B)|= 3$, and $k_{\aut(K)}(B)= 4$. 
\end{itemize}
 In particular, Condition \ref{quasisimples2} is true for $K$.
\end{prop}
\begin{proof}
This can be seen using the GAP Character Table Library. We note that the groups with exceptional Schur multipliers are listed in \cite[Table 6.1.3]{GLS}. 
\end{proof}

\begin{thm}\label{alternating}
Let $p\geq 3$ be prime. Let $K$ be a quasisimple group with $K/\zent{K}\cong \alt_n$,  an alternating group with $n>11$.  Let $B$ be a $p$-block for $K$ with noncyclic, positive defect. Then $k_{\aut(K)}(B)\geq 4$. In particular,
Condition \ref{quasisimples2} is true if $K$ is a covering group for  $S\cong \alt_n$ for $n> 11$. 
\end{thm}
\begin{proof}

The proof here is essentially the same as that of \cite[Proposition 3.4]{MNST}. Let $\hat\alt_n$ and $\hat\sym_n$ denote the double covers, respectively, of $\alt_n$ and $\sym_n$. Recall that $\Aut(S)=\sym_n$ and $\aut(\hat\alt_n)=\hat\sym_n$. Following \cite{Ol93}, a $p$-block of $\sym_n$ has $k(p,w)$ ordinary irreducible characters, and a $p$-block of $\hat\sym_n$ lying over the nontrivial character of $\zent{\hat\sym_n}$ (a ``spin block") has $k^\pm(\bar p, w)$ ordinary irreducible characters, where $w$ is the so-called ``weight" of the block. We remark that our assumption that a defect group is noncyclic forces $w\geq 2$.

  From \cite[(3.11) and Section 13]{Ol93}, we see that these numbers are larger than 6 (and hence there are strictly more than 3 $\aut(K)$-orbits represented in a given block $B$ of $K$) if $p\geq 3$ and $w\geq 2$, except for the case $(p,w)=(3,2)$ and $B$ is a spin block, in which case $k^\pm(\bar{3},2)=6$. In this case, \cite[Proposition 13.19]{Ol93} forces at least one of the characters in the block of $\hat\sym_n$ to restrict to the sum of two characters of $\hat\alt_n$, and hence our block again  contains characters from strictly more than 3 $\aut(K)$-orbits.
\end{proof}

\begin{prop}\label{definingchar}
Condition \ref{quasisimples2} holds for $K$ a quasisimple group with $S:=K/\zent{K}$ of Lie type defined in characteristic $p$ with a non-exceptional Schur multiplier.
\end{prop}
\begin{proof}

We may assume that $K$ is not an exceptional cover of $S:=K/\zent{K}$, as the
latter have been discussed in Proposition \ref{prop:sporadic}.  Now, every $p$-block of $K$ is either maximal defect or defect zero, by \cite[Theorem.]{Hum}.  Hence the defect groups of $B$ are Sylow $p$-subgroups of $K$.
Now, the condition that a Sylow
$p$-subgroup is abelian and non-cyclic forces $S=\PSL_2(p^a)$ for some integer
$a\ge2$, so we may assume that $K=\SL_2(p^a)$ is the Schur covering group
of~$S$. In this situation, the
blocks of maximal defect are in bijection with the characters of $\zent{K}$. 
Namely, we have $B_0(K)$, which contains all members of
$\Irr(K|1_{\zent{K}})\setminus \{\mathrm{St}\}$ and a second block of maximal
defect containing all characters of $K$ that are nontrivial on $\zent{K}$.
(See \cite[Section 5]{Hum}.) By inspection (see \cite[Tab.~2.6]{GM20}) there are
four degrees for characters in $B_0(K)$, and three in the second block of maximal defect.

Hence, it suffices to show that there are two semisimple characters $\chi_s$ of the same degree $q\pm1$ that are not $\aut(K)$-conjugate and nontrivial on $\zent{K}$. (The latter is equivalent to $s\not\in [K^\ast, K^\ast]$ using \cite[Proposition  11.4.12 and Remark 11.4.14]{DM20}). 

Since $a\geq 2$, $p^a-1$ must have at least two distinct divisors, so we consider $x_1, x_2\in C_{p^a-1}$ with these orders. Let $s_i:=\mathrm{diag}(x_i,1)\in\wt{K}^\ast:=\GL_2(p^a)$ for $i=1,2$. Note that $s_i\not\in[\wt{K}, \wt{K}]=\SL_2(p^a)$ . Further, $s_1^\alpha$ cannot be conjugate to $s_2z$ for any $z\in\zent{\wt{K}}$ and $\alpha\in\aut(K)$. Hence the two semisimple characters $\chi_{s_i}$ of $\wt{K}$ for $i=1,2$ cannot be $\aut(K)$-conjugate and restrict to distinct characters of $K$. Hence constituents of these restrictions are not $\aut(K)$-conjugate.
\end{proof}

This leaves us to consider groups $S$ of Lie type in non-defining characteristic. Recall that by Proposition \ref{prop:sporadic}, we may assume that $S$ does not have an exceptional Schur multiplier. Hence the Schur covering group of $S$ is of the form $G=\bG^F$, where $\bG$ is a simple, simply connected algebraic group and $F\colon \bG\rightarrow\bG$ is a Frobenius endomorphism endowing $\bG$  with an $\FF_q$-rational structure, where $p\nmid q$.

Given a semisimple $s\in G^\ast$ of $p'$-order, a fundamental result of Brou{\'e}--Michel shows  that the set  $\mathcal{E}_p(G,s)$ is a union of $p$-blocks of $G$, where $\mathcal{E}_p(G,s)$ is obtained as the union of series $\mathcal{E}(G, st)$ as $t$ runs over elements of $p$-power order in $\cent{G^\ast}{s}$. (See \cite[Theorem 9.12]{CE04}.)

 We first dispense of the Suzuki and Ree groups. 

\begin{prop}\label{suzree}
Condition \ref{quasisimples2} holds when $S=K/Z$ a Suzuki, Ree, or triality group $\tw{2}\type{B}_2(q)$, $\tw{2}\type{G}_2(q)$, $\tw{2}\type{F}_4(q)$, or $\tw{3}\type{D}_4(q)$ with $p\geq 3$ a prime dividing  $|S|$ and not dividing $q$.
\end{prop}
\begin{proof}

Note that the Schur multiplier for $S$ is trivial or $S$ was considered already
in Proposition \ref{prop:sporadic}. Hence, we let $K=S$. Further, $\aut(S)/S$ is cyclic, generated by field automorphisms. For $p\geq 3$ a prime not dividing $q^2$,
the Sylow $p$-subgroups of $S=\tw{2}\type{B}_2(q^2)$ and
$S=\tw{2}\type{G}_2(q^2)$ are cyclic.

So, first let $K=\tw{2}\type{F}_4(q^2)$ with $q^2=2^{2n+1}$. Note that $K^\ast=K$ is self-dual. In this case, the semisimple classes, centralizers, and maximal tori are given in \cite{shinoda}, and the blocks are studied in \cite{malle91}.  

First, suppose that $p\mid (q^2-1)$. Then $K$ has a unique unipotent block (namely, $B_0(K)$) with noncyclic defect group, which contains more than 3 characters of distinct degree. Similarly, there is a unique noncyclic block of positive defect in each series $\mathcal{E}(K, s)$ for $s\in\{t_1, t_2, t_3\}$, with $t_i$ as in \cite{shinoda}, using  \cite[Bem. 1]{malle91}. The remaining blocks of positive defect are cyclic. If $s$ is one of the classes of the form $t_1$ or $t_2$, then this noncyclic block contains two characters from $\mathcal{E}(K, s)$ with distinct degrees. The centralizers $\cent{K}{s}$ contain the maximal torus $\ZZ_{q^2-1}^2$, from which we may obtain $t, t'\in\cent{K}{s}_p$ that are not $\aut(K)$-conjugate (taking, for example, $p$-elements from classes $t_1$ and $t_2$). This yields four characters in the block that are not $\aut(K)$-conjugate, as desired. For $s$ of the form $t_3$, we have $\cent{K}{s}$ is the full maximal torus $\ZZ_{q^2-1}^2$, and for any $t\in\cent{K}{s}_p$, we have $\cent{K}{st}=\cent{K}{s}$. Hence we see that every irreducible character in this block has the same degree. 

When $p\nmid (q^2-1)$, each $\mathcal{E}_p(K, s)$ contains at most one block of positive defect (see \cite[Bem. 1]{malle91}). First, assume $p\mid (q^2+1)$. Here, the noncyclic blocks correspond to $s\in\{t_4, t_5, t_{14}\}$. The set $\mathcal{E}(K, s)$ contains 3, 2, 1 distinct character degrees, respectively, in these cases, and each $\cent{K}{s}$ contains the maximal torus $\ZZ_{q^2+1}^2$. As before, there is only one character degree in the block in the latter case. In the other cases, we argue analogously to the previous paragraph to obtain four characters in the block that are not $\aut(K)$-conjugate.

If instead $p\mid (q^4+1)$, then there are three distinct character degrees in $\mathcal{E}(K, s)$ with $s\in\{t_7, t_9\}$. Then considering any character in $\mathcal{E}(K, st)$ with $t\in\cent{K}{s}_p$, we obtain a fourth character in the block that is not $\aut(K)$-conjugate to these three. If instead $s\in\{t_{12}, t_{13}\}$, we obtain as before that every character in the block has the same degree. The remaining blocks in this case have cyclic defect groups.

Now, let $K=\tw{3}\type{D}_4(q)$. In this case, the blocks have been studied in \cite{deriziotismichler}. Using the results there, we may argue analogously to the situation above.
\end{proof}

In the remaining cases, we would hope to appeal to the strategy employed in \cite[Section 3]{MNST}. Namely, with the above results,  the results of loc. cit. largely reduce the problem of proving Condition \ref{quasisimples2} to the following:

\begin{cond}\label{cond:MNSTorbitfurtherreduction}
Let $\bH$ be a simple, simply connected  reductive group  and $F\colon \bH\rightarrow\bH$ be a Frobenious morphism and $H=\bH^F$ the corresponding finite group of Lie type. Let $B$ be a quasi-isolated $p$-block of $H$ with an elementary abelian defect group $D$. Then
\[k_{\aut(H)}(B)\geq 
\left\{\begin{array}{cc}
4 & \hbox{ if $D$ is not cyclic}\\
3 & \hbox{ if $D$ is  cyclic}\\
\end{array}\right.
\]
\end{cond}

Indeed,  from our above results, we may assume that $S$ is a group of Lie type defined in characteristic distinct from $p$ and that the Schur covering group for $S$ is  $G=\bG^F$ where $\bG$ is a simple, simply connected group with $F$ a Frobenius endomorphism.  Note then that $K$ is a quotient of $G$ by some central subgroup and that, from our assumption that $p\nmid Z$ in Condition \ref{quasisimples2},   \cite[Theorem 9.9(c)]{nbook} tells us that it suffices to prove Condition \ref{quasisimples2} when $K=G/\zent{G}_{p}$, where $\zent{G}_p$ is the Sylow $p$-subgroup of $\zent{G}$.  Our assumption that $p\geq 3$ then means that we may assume that $K=G$ unless $S=\PSL_n^\epsilon(q)$ with $p\mid (q-\epsilon)$ or $S=\type{E}^\epsilon_6(q)$ with $p=3\mid (q-\epsilon)$. 

Then indeed, when $K=G$, Condition \ref{cond:MNSTorbitfurtherreduction} implies Condition \ref{quasisimples2} by \cite[Proposition 3.9 and Lemma 3.12]{MNST} (see also Remarks 3.10, 3.11 of Loc. Cit.) and the fact that Bonnafe--Rouquier correspondence preserves isomorphism types of defect groups when the defect is abelian, by \cite[Theorem 7.16]{KM13}. Of course, in the cases of $S=\PSL_n^\epsilon(q)$ with $p\mid (q-\epsilon)$ or $S=\type{E}^\epsilon_6(q)$ with $p=3\mid (q-\epsilon)$, some additional work is needed, as was the case in \cite[Prop. 3.16(b) and Theorem 3.21]{MNST}. 

However, this method is not quite sufficient for completing the proof of Condition \ref{quasisimples2}. 
Unfortunately, although we see by following the proofs in \cite[Sections 3.4-3.6]{MNST} that  Condition \ref{cond:MNSTorbitfurtherreduction} holds in many situations, it turns out that there are indeed cyclic quasi-isolated blocks with $k_{\aut(G)}(B)=2$. (This is already pointed out in \cite{MNST} after the statement of   Theorem 3.1.) We list some additional situations in the following:

\begin{exmp}\label{ex:typeAexceptions}
Let $G=\SL_n^\epsilon(q)$ and let $\wt{G}=\GL_n^\epsilon(q)$.  Let $p$ be an odd prime not dividing $q$ and let $B$ be a $p$-block of $G$ with positive defect. The following situations lead to exceptions to Condition \ref{cond:MNSTorbitfurtherreduction}:
\begin{enumerate}[label=(\Roman*)]

\item If $B$ is noncyclic:
\begin{itemize}
\item $n=3$, $p=5\mid\mid(q-\epsilon)$, and $B$ lies under a block $\wt{B}$ of $\wt{G}=\GL_3^\epsilon(q)$ indexed by a semisimple $5'$-element $\wt{s}\in \wt{G}$  with $\cent{\wt{G}}{\wt{s}}\cong C_{q-\epsilon}^3$. In this case, $k_{\aut(G)}(B)\geq 3$ (and equality can occur) and $|\mathrm{cd}({\wt{B}})|=1$. Note: this includes the quasi-isolated block of $\SL_3(q)$ under the block indexed by the semisimple element $\mathrm{diag}(1,\zeta_3, \zeta_3^{-1})$, where $|\zeta_3|=3$. (See also Remark \ref{rem:counterexample}.)
\item $n=3$, $p=3\mid\mid(q-\epsilon)$, and $B$ lies under a block $\wt{B}$ of $\wt{G}=\GL_3^\epsilon(q)$ indexed by a semisimple $3'$-element $\wt{s}\in \wt{G}$  with $\cent{\wt{G}}{\wt{s}}\cong C_{q-\epsilon}^3$. In this case, $k_{\aut(G)}(B)\geq 3$ (and equality can occur), $|\zent{G}|=3$, and the block of $S$ contained in $B$ is cyclic. 
\item $p=3\mid\mid(q+\epsilon)$ and $B$ lies under a block $\wt{B}$ of $\wt{G}=\GL_n^\epsilon(q)$ with defect group $C_{3}^2 \leq C_{q+\epsilon}^2$.
\end{itemize}
\item If $B$ is cyclic:
\begin{itemize}
\item $n=2$, $p\mid\mid(q-\epsilon)$, and $B$ lies under a block $\wt{B}$ of $\wt{G}=\GL_2^\epsilon(q)$ indexed by a semisimple $p'$-element $\wt{s}\in \wt{G}$  with $\cent{\wt{G}}{\wt{s}}\cong C_{q-\epsilon}^2$. Here $k_{\aut(G)}(B)\geq 2$ (and equality can occur). Note: this includes the quasi-isolated block of $\SL_2(q)$ under the block indexed by the semisimple element $\mathrm{diag}(-1,1)$.
\item $B$ lies under a block $\wt{B}$ of $\wt{G}$ indexed by a semisimple $p'$-element with $\cent{\wt{G}}{\wt{s}}\cong C_{q^\delta-\eta}$ where $\delta=n$ and $p\mid\mid(q^\delta-\eta)$.  Here $k_{\aut(G)}(B)\geq 2$ (and equality can occur).

\end{itemize}
\end{enumerate}

\end{exmp}

\begin{rem}\label{rem:counterexample}
We remark that some of the exceptions given in Example \ref{ex:typeAexceptions} 
mean that Condition \ref{cond:MNSTorbitfurtherreduction} is not always feasible. Further, the first  exception of (I) yields examples of $5$-blocks with $k_{\aut(S)}(B)=3$, meaning that Condition \ref{quasisimples2} will also not always hold. For example, let $K=\SL_3(q)$ where $q=q_0^4$, $q_0\equiv 2\pmod 3$, $q_0\equiv 3\pmod 5$, and let $p=5\mid\mid(q-1)$. Let $\zeta\in\FF_q^\times$ with $|\zeta|=3$. Then a block $B$ of $K$ lying below the (unique) block $\wt{B}$ of $\GL_3(q)$ in $\mathcal{E}_5(\GL_3(q), s)$ with $s=\mathrm{diag}(\zeta, \zeta^{-1}, 1)$ satisfies $k_{\aut(S)}(B)=3$ and $|D|=C_5\times C_5$. Further, $|\mathrm{cd}(B)|=2$ although $|\mathrm{cd}(\wt{B})|=1$.
\end{rem}

\bigskip

We also remark that it is still not known whether there are $3$-blocks with 3 irreducible characters with defect group $\CCC_3\times\CCC_3$. This problem appeared first in \cite{Kiy84} and has come up often. See, for instance, \cite[p. 677]{AS22}. Our proof of Theorem C shows that if one could prove Condition B with the additional hypothesis that $D=\CCC_3\times \CCC_3$ then this problem would have a negative answer. On the other hand, a proof of Condition B when $|D|>C$ for some universal constant $C$, would settle the case $k(B)=3$ of Brauer's Problem 21 for arbitrary groups.

\end{document}